\documentclass[a4paper,12pt]{amsart}
\usepackage{amsmath,amssymb,enumerate,epsfig,xcolor}

\newtheorem{thm}{Theorem}
\newtheorem{cor}[thm]{Corollary}
\newtheorem{prop}[thm]{Proposition}
\newtheorem{lemma}[thm]{Lemma}
\newtheorem{q}[thm]{Question}
\newtheorem{Rmk}[thm]{Remark}

\newtheorem*{rmk}{Remark}
\newtheorem*{rmks}{Remarks}
\newenvironment{enui}{\begin{enumerate}[(i)]}{\end{enumerate}}
\newcommand\reff[1]{(\ref{#1})}
\newcommand\Reff[1]{\textbf{(\ref{#1})}}
\newcommand\wt[1]{{\widetilde{#1}}}
\newcommand{\N}{\mathbb{N}}
\newcommand{\Z}{\mathbb{Z}}
\newcommand{\R}{\mathbb{R}}
\newcommand\Om{\Omega}
\newcommand\om{\omega}
\newcommand\x{\times}
\newcommand\Si{\Sigma}

\renewcommand\phi{\varphi}
\newcommand\ka{\kappa}
\newcommand\lam{\lambda}

\newcommand{\BAR}[1]{{\overline{#1}}}
\newcommand\wo{\setminus}

\newcommand\id{{\operatorname{id}}}
\newcommand\nn{{\nonumber}}
\newcommand\sub{\subseteq}
\newcommand\C{\mathbb C}
\newcommand\dd{\partial}

\title[A symplectic embedding of the cube with minimal sections]{A symplectic embedding of the cube with minimal sections and a question by Schlenk}
\author{Fabian Ziltener}

\begin{document}

\maketitle

\begin{abstract} I prove that the open unit cube can be symplectically embedded into a longer polydisc in such a way that the area of each section satisfies a sharp bound and the complement of each section is path-connected. This answers a variant of a question by F.~Schlenk.
\end{abstract}

\section{The main result}
Let $n\geq2$. By $q^1,p_1,\ldots,q^n,p_n$ we denote the standard coordinates in $\R^{2n}$, and we equip $\R^{2n}$ with the standard symplectic form $\om_0:=\sum_{i=1}^ndq^i\wedge dp_i$.\footnote{Following the physicists' convention I use an upper index for the $i$-th coordinate of a point $q$ in the base manifold $\R^n$ and lower index for the $i$-th coordinate of a covector $p\in\R^n=T_q^*\R^n$.} %end footnote
We denote by $B^m_r$ resp.~$\BAR B^m_r$ the open resp.~closed %reason for closed: question by McDuff asked for closed ball
ball in $\R^m$ of radius $r$ around 0. M.~Gromov's famous nonsqueezing theorem \cite[Corollary, p.~310]{Gro} implies %don't say states, since it is formulated for an open ball
that $\BAR B^{2n}_r$ does not symplectically embed into the closed unit symplectic cylinder $\BAR B^2_1\x\R^{2n-2}$ if $r>1$. In \cite{SchlenkQ} F.~Schlenk investigated how flexible symplectic embeddings are in the case $r\leq1$. More precisely, for every $z\in\R^{2n-2}$ we define
\[\iota_z:\R^2\to\R^{2n},\quad\iota_z(y):=(y,z).\]
Answering a question by D.~McDuff \cite{McDuff}, in \cite[Theorem 1.1]{SchlenkQ} Schlenk proved that for every $a>0$ there exists a symplectic embedding $\phi$ of $\BAR B^{2n}_1$ into $\BAR B^2_1\x\R^{2n-2}$, such that for every $z\in\R^{2n-2}$ the section $\iota_z^{-1}\big(\phi\big(\BAR B^{2n}_1\big)\big)$ has area\footnote{This means two-dimensional Lebesgue measure.} at most $a$. %\cite{SchlenkQ}: symplectic embedding of a subset of $\R^{2n}$: extends to a symplectic embedding of an open neighbourhood 

Schlenk's lifting method \cite[Section 8.4]{SchlenkE} also shows that for every positive integer $k$ and every $a>\frac1k$ there exists a symplectic embedding of the open cube $(0,1)^{2n}$ into the open polydisc $(0,1)^{2n-1}\x(0,k)$, whose sections have area at most $a$. %Schlenk didn't state this, but we obtain this by slicing the square into enough strips and moving some strips to the same height
The main result of the present article answers the following two questions:
\begin{q}\label{q:c} Is this statement true with the integer $k$ replaced by a general real number $c\geq1$?
\end{q}
\begin{q}\label{q:a sharp} Can the bound $a$ on the areas of the sections be made sharp, i.e., equal to $\frac1c$?\footnote{There is always a section of area at least $\frac1c$, by Fubini's theorem. Hence $a=\frac1c$ is the minimal possible bound.}%end footnote
\end{q}
I also answer a variant of the following question by Schlenk. For every bounded subset $S$ of $\R^m$ we define the \emph{bounded hull of $S$} to be the union of $S$ and all bounded connected components of $\R^m\wo S$.
\begin{q}[Schlenk, \cite{SchlenkQ}, Question 2.2]\label{q:Schlenk} Let $n\geq2$, $\phi$ be a symplectic embedding of $B^{2n}_1$ into $B^2_1\x\R^{2n-2}$, and $a<\pi$. Does there exist $z\in\R^{2n-2}$ such that the bounded hull of the closure of the section $\iota_z^{-1}\big(\phi\big(B^{2n}_1\big)\big)$ has area %the section is measurable
at least $a$?
\end{q}
The main result of this article is the following.
\begin{thm}\label{thm:sec} For every $n\geq2$ and $c\in[1,\infty)$ there exists a symplectic embedding $\phi:(0,1)^{2n}\to(0,1)^{2n-1}\x(0,c)$, such that for every $z\in\R^{2n-2}$ the following holds:
\begin{enui}
\item\label{thm:sec:area} The section $\iota_z^{-1}\big(\phi\big((0,1)^{2n}\big)\big)$ has area at most $\frac1c$. %0 for some points
\item\label{thm:sec:compl} Its complement in $\R^2$ is path-connected.
\end{enui}
\end{thm}
This theorem answers Questions \ref{q:c} and \ref{q:a sharp} affirmatively. It also provides a negative answer to Schlenk's Question \ref{q:Schlenk} with the word ``closure'' dropped. It even implies that there exists a symplectic embedding for which the bounded hull of each section has arbitrarily small area: %We cannot deduce from the theorem that the section equals its bounded hull, since we look at a subset of the cube.
\begin{cor}\label{cor:bounded hull} For every $n\geq2$ and $a>0$ there exists a symplectic embedding $\psi:B^{2n}_1\to B^2_1\x\R^{2n-2}$, such that the bounded hull of each section of $\psi\big(B^{2n}_1\big)$ has area at most $a$.
\end{cor}
(For a proof see p.~\pageref{proof:cor:bounded hull}.) This corollary is optimal in the sense that its statement becomes false if we replace $B^{2n}_1$ and $B^2_1$ by the \emph{closed} balls $\BAR B^{2n}_1$ and $\BAR B^2_1$. Even the following is true:
\begin{prop}[F.~Lalonde, D.~McDuff]\label{prop:closed} Let $n\in\N$ and $\phi:\BAR B^{2n}_1\to\BAR B^2_1\x\R^{2n-2}$ be a symplectic embedding.\footnote{We don't impose any restrictions on how $\phi$ maps the boundary of the ball.} Then there exists $z\in\R^{2n-2}$, such that the section $\iota_z^{-1}\big(\phi\big(\BAR B^{2n}_1\big)\big)$ contains the circle of radius $1$ around 0.
\end{prop}
In particular the bounded hull of this section equals $\BAR B^2_1$, which has area $\pi$.
\begin{proof}[Proof of Proposition \ref{prop:closed}] This follows from \cite[Lemma 1.2]{LM}.
\end{proof}
\begin{rmk}Let $\phi$ be as in the statement of Theorem \ref{thm:sec}. Then each section of the image of $\phi$ equals its own bounded hull. Hence $\phi$ is a \emph{sharp} counterexample to a variant of Question \ref{q:Schlenk} concerning embeddings of cubes. %with ``closed'' dropped
\end{rmk}
In the case $n=2$ the idea of proof of Theorem \ref{thm:sec} is to consider the linear symplectic map $\Psi:(q,p)\mapsto(Q,P)$, induced by the Lagrangian shear $p\mapsto P:=\big(p_1,cp_1+p_2\big)$. The $P_2$-sections of the image of the square $(0,1)^2$ under this shear have length at most $\frac1c$. Hence the area of each section of $\Psi\big((0,1)^4\big)$ is at most $\frac1c$. To make the image of $\Psi$ fit in the polydisc $(0,1)^3\x(0,c)$, we wrap its upper part (in $P_2$-direction) back to the lower part, by passing to the quotient $\R/c\Z$. We also wrap the $Q^1$-coordinate. See Figure \ref{fig:shear}. 
\begin{figure}
\centering
\leavevmode\epsfbox{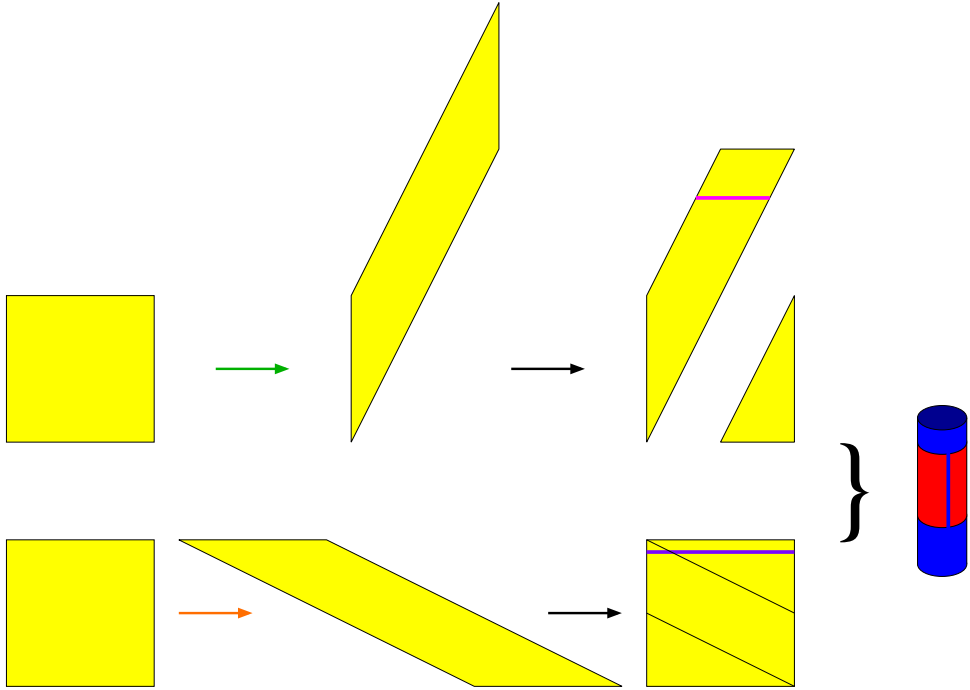}
\caption{The \textcolor{green}{green} arrow depicts the Lagrangian shear $p\mapsto P:=\big(p_1,cp_1+p_2\big)$, and the \textcolor{orange}{orange} arrow the induced shear in the $q$-plane. The black arrows depict the wrapping maps. The \textcolor{magenta}{magenta} line segment is a $\BAR P_2$-section of the image of the square under the composed map in the $p$-plane, where $\BAR P_2\in\R/(c\Z)$. The \textcolor{violet}{violet} set depicts a $\BAR Q^2$-section of the image of the open square under the composed map in the $q$-plane, where $\BAR Q^2\in\R/\Z$. The bracket $\}$ indicates that the product of these two sets is given by the \textcolor{red}{red} ribbon on the blue cylinder. The image of this ribbon under some area-preserving map is a section the image of the desired symplectic embedding $\phi$. It has area equal to $\frac1c$.}\label{fig:shear}
\end{figure}
Finally, we compose the resulting map with the product of two area preserving embeddings of finite cylinders into rectangles. This yields a symplectic embedding with the desired properties.
\begin{rmks}[method of proof, related result, terminology]\begin{itemize}
\item This construction is similar to L.~Traynor's symplectic wrapping construction, which she used e.g.~to show that certain polydiscs embed into certain cubes, see \cite{Tra} and \cite[Chapter 7]{SchlenkE}. One difference is that I wrap coordinates of mixed type ($Q$ and $P$), whereas Traynor wraps coordinates of pure type.
\item Schlenk proved a nonsharp result regarding the areas of the bounded hulls of the sections. %Really: closure, but this doesn't make any difference in his situation.
More precisely, his folding method \cite[Section 8.3]{SchlenkE} can be used to prove that for every $n\geq2$, positive integer $k$, and $\ell\in(0,1)$ there exists a symplectic embedding $\phi:(0,\ell)^{2n}\to(0,1)^{2n-1}\x(0,k)$, such that the bounded hull of every section of $\phi\big((0,\ell)^{2n}\big)$ has area at most $\frac1k$. Theorem \ref{thm:sec} improves this in the following ways:
\begin{itemize}
\item It treats the critical case $\ell=1$.
\item It makes the area estimate sharp.
\item It holds for any real number $c\geq1$, not only for an integer $c=k$.
\item The proof of Theorem \ref{thm:sec} is easier than the folding method.
\end{itemize}
\item In \cite{SchlenkQ} and \cite[p.~226]{SchlenkE} Schlenk calls the bounded hull of the closure of a set its ``simply connected hull''. The simply connected hull of a simply connected compact subset $S$ of $\R^m$ need not be equal to $S$. In the case $m\geq3$ an example is given by the sphere $S:=S^{m-1}$, and in the case $m=2$ by the Warsaw circle. This set is produced by closing up the topologist's sine curve with an arc. For this reason I prefer the terminology ``bounded hull''. Since this notion is only defined for \emph{bounded} subsets of $\R^m$, no confusion should arise from the fact that the bounded hull of a bounded set $S$ can differ from $S$.
\item For more information about related work see \cite{SchlenkE}.
\end{itemize}
\end{rmks}
\section{Proofs of the main result and of Corollary \ref{cor:bounded hull}}
In the proofs of Theorem \ref{thm:sec} and Corollary \ref{cor:bounded hull} we will use the following lemma.
\begin{lemma}[squaring the disc and the cylinder]\label{le:square} We denote $r:=\pi^{-\frac12}$.
\begin{enui}
\item\label{le:square:chi 1} There exists a homeomorphism
\[\ka:\BAR B^2_r\to[0,1]^2,\]
that restricts to a (smooth) symplectomorphism between the interiors.
\item\label{le:square:chi} For every $y_0\in(0,1)^2$ there exists continuous map
\[\lam:(\R/\Z)\x[0,1]\to[0,1]^2\]
that maps $(\R/\Z)\x\{1\}$ to $y_0$, and restricts to a homeomorphism from $(\R/\Z)\x[0,1)$ to $[0,1]^2\wo\{y_0\}$ and to a symplectomorphism from $(\R/\Z)\x(0,1)$ to $(0,1)^2\wo\{y_0\}$.
\end{enui}
\end{lemma}
The idea of proof of this lemma is explained by Figure \ref{fig:cyl disc square}. %
\begin{figure}
\centering
\leavevmode\epsfbox{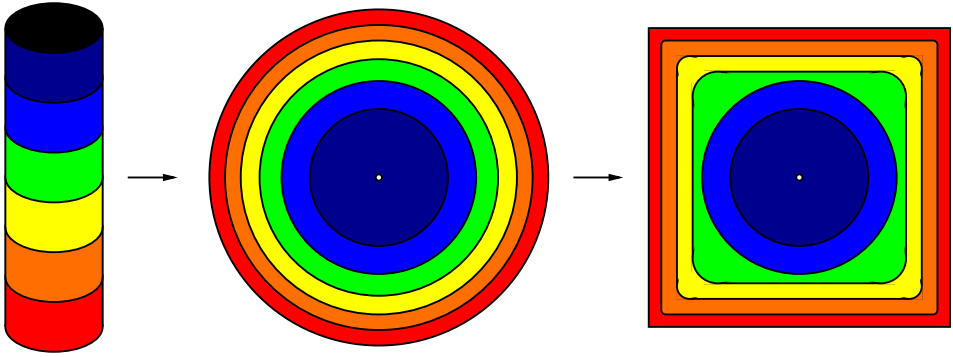}
\caption{The two arrows depict area-preserving smooth embeddings whose composition is an area-preserving embedding of the open cylinder into the open square. The idea of proof of Lemma \ref{le:square} is to choose such maps in such a way that they continuously extend to the closed cylinder and the closed disc, respectively.}\label{fig:cyl disc square}
\end{figure}
In the proof of Lemma \ref{le:square} we will use the following.
\begin{Rmk}[straightening corners]\label{rmk:square} We denote by $\Si$ the square $[0,1]^2$ without the corners. Let $r>0$ and $S$ be a subset of the circle of radius $r$ consisting of four points. There exists homeomorphism $\theta:[0,1]^2\to\BAR B^2_r$ that restricts to a diffeomorphism from $\Si$ onto $\BAR B^2_r\wo S$, such that $(\theta|\Si)_*\om_0$ extends to a nonvanishing smooth 2-form on $\BAR B^2_r$. 

To see this, observe that the map 
\[\wt\theta:[0,\infty)^2\to\R\x[0,\infty),\quad\wt\theta(z):=\frac{z^2}{|z|},\]
is a homeomorphism that restricts to a diffeomorphism from $[0,\infty)^2\wo\{0\}$ onto $\big(\R\x[0,\infty)\big)\wo\{0\}$, that satisfies
\[\big(\wt\theta\big|[0,\infty)^2\wo\{0\}\big)_*\om_0=\frac{\om_0}2.\]
The desired map $\theta$ can be constructed from four copies of $\wt\theta$ (one for each corner), using charts for $\BAR B^2_r$ and a cut off argument.
\end{Rmk}
\begin{prop}[Banyaga's Moser stability with boundary]\label{prop:Moser} Let $M$ be a compact connected oriented smooth manifold, and $\Om_0,\Om_1$ volume forms on $M$ satisfying
\[\int_M\Om_0=\int_M\Om_1.\]
Then there exists diffeomorphism $\phi$ of $M$ satisfying
\[\phi_*\Om_0=\Om_1,\quad\phi|\dd M=\id.\]
\end{prop}
\begin{proof} See \cite[Th\'eor\`eme, p.~127]{Ban}.
\end{proof}
\begin{proof}[Proof of Lemma \ref{le:square}] To prove \Reff{le:square:chi 1}, we define $r:=\pi^{-\frac12}$ and choose a map $\theta$ as in Remark \ref{rmk:square}. We define
\[M:=\BAR B^2_r,\quad\Om_0:=\theta_*\om_0,\quad\Om_1:=\om_0.\]
We have
\[\int_M\Om_0=\int_\Si\om_0=1=\int_M\Om_1.\]
Hence the hypotheses of Proposition \ref{prop:Moser} are satisfied. We choose a diffeomorphism $\phi$ as in the statement of this proposition. The map
\[\ka:=(\phi\circ\theta)^{-1}:\BAR B^2_r\to[0,1]^2\]
has the required properties.

We prove \Reff{le:square:chi}. There exists a symplectomorphism
\[\chi:(\R/\Z)\x[0,1)\to\BAR B^2_r\wo\{0\}.\]
For example, consider $y:\R/\Z\to\C=\R^2$, $y(\bar q):=e^{2\pi iq}$, where $q\in\bar q$ is an arbitrary representative, and define
\[\chi(\BAR q,p):=r\sqrt{1-p}y(\BAR q).\]

We choose a symplectomorphism $\xi$ of $[0,1]^2$ that equals the identity in a neighbourhood of the boundary and maps $\ka(0)$ to $y_0$.\footnote{$\xi$ is a smooth map in the sense of manifolds with boundary and corners.} We obtain such a map as the Hamiltonian flow of a suitable function on $(0,1)^2$ with compact support. The map
\[\lam:=\left\{\begin{array}{ll}
\xi\circ\ka\circ\chi&\textrm{ on }(\R/\Z)\x[0,1),\\
y_0&\textrm{ on }(\R/\Z)\x\{1\}
\end{array}\right.\]
has the required properties. This proves \reff{le:square:chi} and completes the proof of Lemma \ref{le:square}.
\end{proof}
\begin{proof}[Proof of Theorem \ref{thm:sec}] Consider the case $n=2$. We denote by
\[\pi:\R^4\to(\R/\Z)\x\R\x\R\x(\R/c\Z)\]
the canonical projection, and equip $(\R/\Z)\x\R\x\R\x(\R/c\Z)$ with the symplectic form induced by $\om_0$ and $\pi$. We denote $y_0:=z_0:=\left(\frac12,\frac12\right)$. We choose a map $\lam$ as in Lemma \ref{le:square}\reff{le:square:chi}. It follows from the same lemma that there exists a symplectomorphism
\[\lam':(0,1)\x(\R/c\Z)\to\big((0,1)\x(0,c)\big)\wo\{z_0\}.\]%We need to switch the coordinates in the domain.
We define
\begin{eqnarray*}&\Psi:\R^4\to\R^4,\quad\Psi\big(q^1,p_1,q^2,p_2\big):=\big(q^1-cq^2,p_1,q^2,cp_1+p_2\big),&\\
&\phi:=(\lam\times\lam')\circ\pi\circ\Psi\big|(0,1)^4.&
\end{eqnarray*}
The map $\phi$ is well-defined, since $\pi\circ\Psi$ maps $(0,1)^4$ to the product of the domains of $\lam$ and $\lam'$. The map $\phi$ is a symplectic immersion, as it is the composition of three symplectic immersions. A straight-forward argument shows that $\pi\circ\Psi\big|(0,1)^4$ is injective. Since $\lam|(\R/\Z)\x(0,1)$ and $\lam'$ are injective, it follows that the same holds for $\phi$. Hence $\phi$ is a symplectic embedding of $(0,1)^4$ into $(0,1)^3\x(0,c)$.

Let $(Q^2,\BAR P_2)\in(0,1)\x(\R/c\Z)$. We have
\begin{align}\nn U_{Q^2,\BAR P_2}:=&\big\{(\BAR Q^1,P_1)\in(\R/\Z)\x(0,1)\,\big|\,\big(\BAR Q^1,P_1,Q^2,\BAR P_2\big)\in\pi\circ\Psi\big((0,1)^4\big)\big\}\\
\label{eq:V W}=&V_{Q^2}\x W_{\BAR P_2},
\end{align}
\begin{equation}\label{eq:V}V_{Q^2}:=\big\{q^1-cQ^2+\Z\,\big|\,q^1\in(0,1)\big\}=(\R/\Z)\wo\{-cQ^2+\Z\},\end{equation}
\begin{align}W_{\BAR P_2}&:=\big\{P_1\in(0,1)\,\big|\,\exists p_2\in(0,1):\,cP_1+p_2+c\Z=\BAR P_2\big\}\\
\nn&=(0,1)\cap\bigcup_{p_2\in(0,1)}\frac{\BAR P_2-p_2}c,
\end{align}
where $\frac{\BAR P_2-p_2}c\in\R/\Z$. The set $W_{\BAR P_2}$ is an open subinterval of $(0,1)$ or the union of two such subintervals. It has length $\frac1c$. Using \reff{eq:V W} and \reff{eq:V}, it follows that $U_{Q^2,\BAR P_2}$ has area equal to $\frac1c$. Figure \ref{fig:U} depicts the set $U_{Q^2,\BAR P_2}$.
\begin{figure}
\centering
\leavevmode\epsfbox{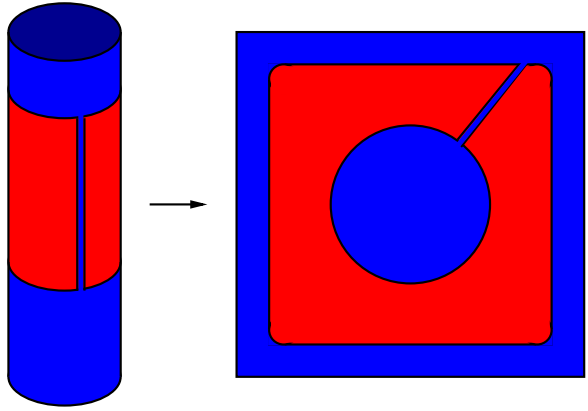}
\caption{The arrow depicts the area preserving map $\lam:(\R/\Z)\x(0,1)\to(0,1)^2$. (Compare to Figure \ref{fig:cyl disc square}.) It sends the upper part of the cylinder close to the center of the disc. The red ribbon on the cylinder is $U_{Q^2,\BAR P_2}$, the section of the image of $\phi$. The point $\BAR P_2$ determines the height of the upper boundary of the red ribbon, and therefore the radius of the circle inside the square. The point $Q^2$ determines the position of the blue slit. Because of this slit, the blue set on the right is path-connected. This is the complement of the image of the section under the map $\lam$.}
\label{fig:U} 
\end{figure} 

Let now $z\in\big((0,1)\x(0,c)\big)\wo\{z_0\}$. We denote $(Q^2,\BAR P_2):={\lam'}^{-1}(z)$. We have
\begin{equation}\label{eq:psi -1}\lam^{-1}\left(\iota_z^{-1}\big(\phi\big((0,1)^4\big)\big)\right)=U_{Q^2,\BAR P_2}.\end{equation}
Since $\lam$ is area-preserving, it follows that the section $\iota_z^{-1}\big(\phi\big((0,1)^4\big)\big)$ has area equal to $\frac1c$. For $z=z_0$ or $z$ outside of $(0,1)\x(0,c)$, the section is empty. This proves \reff{thm:sec:area}.

To prove property \reff{thm:sec:compl}, consider the continuous path 
\[y:[0,1]\to[0,1]^2,\quad y(t):=\lam\big(-cQ^2+\Z,t\big).\]
The point $y(0)$ lies on the boundary of the square $[0,1]^2$. It follows from \eqref{eq:V} that the path $y$ lies inside the complement of $\iota_z^{-1}\big(\phi\big((0,1)^4\big)\big)$ in $\R^2$. Every point outside $(0,1)^2$ can be connected to $y(0)$ through a continuous path outside of $(0,1)^2$. Every point in the complement of $\iota_z^{-1}\big(\phi\big((0,1)^4\big)\big)$ in $(0,1)^2$ can be connected to a point on the path $y$ through a path in this complement. This follows from \eqref{eq:psi -1} and the facts $U_{Q^2,\BAR P_2}=V_{Q^2}\x W_{\BAR P_2}$, $V_{Q^2}=(\R/\Z)\wo\{-cQ^2+\Z\}$. See again Figure \ref{fig:U}. This proves \reff{thm:sec:compl}.

Hence $\phi$ has the desired properties. This proves Theorem \ref{thm:sec} in the case $n=2$. For $n\geq3$ we take the product of $\phi$ with the identity map.
\end{proof}
In the proof of Corollary \ref{cor:bounded hull} we will use the following.
\begin{Rmk}[monotonicity]\label{rmk:mon} The bounded hull is monotone in the sense that if $A\sub B\sub\R^m$ are bounded sets then the bounded hull of $A$ is contained in the bounded hull of $B$.
\end{Rmk}
\begin{proof}[Proof of Corollary \ref{cor:bounded hull}]\label{proof:cor:bounded hull} We define $r:=\pi^{-\frac12}$. By a rescaling argument it suffices to show that for every $a\in(0,1]$ there exists a symplectic embedding $\psi:B^{2n}_r\to B^2_r\x\R^{2n-2}$, such that the bounded hull of each section of $\psi(B^{2n}_r)$ has area at most $a$. To prove this statement, we choose $\phi$ is as in the conclusion of Theorem \ref{thm:sec} with $c:=\frac1a$. We choose a map $\ka$ as in Lemma \ref{le:square}\reff{le:square:chi 1}. The map
\[\psi:=(\ka^{-1}\x\id)\circ\phi\circ\big(\ka\x\cdots\x\ka\big):B^{2n}_r\to B^2_r\x\R^{2n-2}\]
is a symplectic embedding. Let $z\in\R^{2n-2}$. Property \reff{thm:sec:compl} in Theorem \ref{thm:sec} implies that the complement of $V:=\ka^{-1}\Big(\iota_z^{-1}\big(\phi\big((0,1)^{2n}\big)\big)\Big)$ in $\R^2$ is path-connected. Hence $V$ equals its bounded hull. The section $\iota_z^{-1}\big(\psi(B^{2n}_r)\big)$ is contained in $V$. Using Remark \ref{rmk:mon}, it follows that the bounded hull of this section is also contained in $V$. Using Theorem \ref{thm:sec}\reff{thm:sec:area} and that $\ka$ is area-preserving, it follows that this bounded hull has area at most $\frac1c=a$. Hence $\psi$ has the desired properties. This proves Corollary \ref{cor:bounded hull}.
\end{proof}
\section{Acknowledgments}
I would like to thank Felix Schlenk for an interesting discussion and for proof-reading the first version of this article.
\bibliographystyle{amsalpha}
\bibliography{amsj,references}
\end{document}